\documentclass[12pt,a4paper]{article}

\usepackage{amssymb, amsmath, amsthm}

\usepackage[cm]{fullpage}
\usepackage[english]{babel}
\usepackage[T1]{fontenc}
\usepackage[pdftex]{graphicx}
\usepackage{booktabs}
\usepackage{xcolor}
\usepackage{mathrsfs}

\usepackage{float}
\DeclareGraphicsExtensions{.png,.pdf}
\graphicspath{{./images/}}

\usepackage{hyperref}
\usepackage{autonum}

\newtheorem{thm}{Theorem}
\newtheorem{lem}{Lemma}
\newtheorem{prop}{Proposition}
\newtheorem{rem}{Remark}

\title{Collocation method for a functional equation arising in behavioral sciences\footnote{This is the accepted version of a paper published in Journal of Computational and Applied Mathematics 458 (2025), 116343 with DOI: \url{https://doi.org/10.1016/j.cam.2024.116343}}}
\author{
Josefa Caballero\thanks{Departamento de Matem\'aticas, Universidad de Las Palmas de Gran Canaria, Campus de Tafira Baja, $35017$ Las Palmas de Gran Canaria, Spain.}, \and Hanna Okrasi{\'n}ska-P{\l}ociniczak\thanks{Department of Mathematics, Wroclaw University of Environmental and Life Sciences, ul. C.K. Norwida 25, 50-275 Wroclaw, Poland}, \and \L ukasz P{\l}ociniczak\thanks{Faculty of Pure and Applied Mathematics, Wroc{\l}aw University of Science and Technology, Wyb. Wyspia\'nskiego 27, 50-370 Wroc\l caw, Poland, \underline{corresponding author:} \texttt{lukasz.plociniczak@pwr.edu.pl}}, \and Kishin Sadarangani$^{\ast}$ 
}
\date{}

\begin{document}
\maketitle

\begin{abstract}
We consider a nonlocal functional equation that is a generalization of the mathematical model used in behavioral sciences. The equation is built upon an operator that introduces a convex combination and a nonlinear mixing of the function arguments. We show that, provided some growth conditions of the coefficients, there exists a unique solution in the natural Lipschitz space. Furthermore, we prove that the regularity of the solution is inherited from the smoothness properties of the coefficients.

As a natural numerical method to solve the general case, we consider the collocation scheme of piecewise linear functions. We prove that the method converges with the error bounded by the error of projecting the Lipschitz function onto the piecewise linear polynomial space. Moreover, provided sufficient regularity of the coefficients, the scheme is of the second order measured in the supremum norm. 

A series of numerical experiments verify the proved claims and show that the implementation is computationally cheap and exceeds the frequently used Picard iteration by orders of magnitude in the calculation time. \\


\end{abstract}

\section{Introduction}
We investigate a functional equation of the following form
\begin{equation}\label{eqn:MainEq0}
u(x) = \varphi(x) u(\varphi_1(x)) + (1-\varphi(x)) u(\varphi_2(x)) + f(x),
\end{equation}
where $\varphi$, $\varphi_{1},\varphi_{2}$, and $f$ are known coefficients satisfying certain growth and structure conditions to be given below. Note that, excluding the trivial case where $\varphi_{1},\varphi_{2}$ are linear functions, the equation can be thought of as nonlocal in which the argument of the sought solution is mixed in a possibly nonlinear way. From another point of view, the above can be seen as a functional equation with two vanishing delays \cite{liu1995linear} (in contrast with the usual additive, or nonvanishing, delay of the form $x - \tau$ for some $\tau > 0$). This feature of the problem introduces some difficulties in both analytical and numerical treatment. For example, probably the simplest case where the delay is proportional is the functional pantograph equation \cite{iserles1993generalized}
\begin{equation}
u(x) = a u(b x) + f(x), \quad a\in\mathbb{R}, \quad 0<b<1, 
\end{equation}
which already possesses a very rich structure and requires nontrivial techniques to analyse that attracted many researchers through several last decades \cite{buhmann1993stability}. There is also a broad interest in Volterra integral equations with nonvanishing delays \cite{brunner2017volterra} and its stochastic generalizations \cite{yang2019mean}. One of the motivations to investigate \eqref{eqn:MainEq0} is its emergence in behavioral sciences as a model of learning processes of various species \cite{istruactescu1976functional}. The simplest one refers to the paradise fish \cite{turab2020corrigendum} and is based on the experiment \cite{bush1956two}. In this model, the coefficients have the form 
\begin{equation}\label{eqn:ParadiseFishCoeff}
\varphi(x) = x, \quad \varphi_1(x) = 1-\alpha+\alpha x, \quad \varphi_2(x) = \beta x, \quad f(x) \equiv 0, \quad 0<\alpha\leq\beta<1,
\end{equation}
and boundary conditions $u(0) = 0$, $u(1) = 1$. Other generalizations to different species and learning processes were discussed, for example, in \cite{epstein1966difference}. 

The equation similar to \eqref{eqn:MainEq0} but in the case of the paradise fish model \eqref{eqn:ParadiseFishCoeff} was initially explored in \cite{lyubich1973functional}, employing the Schauder fixed point theorem to establish the existence of solutions. A crucial stipulation was introduced, necessitating that the solution be expressible through a specific power series. This matter was investigated in additional studies in \cite{istruactescu1976functional}, where the author demonstrated the existence and uniqueness using the Banach contraction principle and Picard's iteration. Subsequently, the research into the existence and uniqueness results was expanded in \cite{berinde2015functional}. In our previous paper \cite{paradiseFefiKishinLukas} we investigated a general situation \eqref{eqn:MainEq0} with vanishing source and non-zero boundary conditions. We have proved the existence and uniqueness of the solution under some growth conditions on the coefficients. We have also noticed that, to the best of our knowledge, the only available numerical method for obtaining approximate solutions, Picard's iteration, is impractical and requires long computation times. To aid in this, we have also proposed some analytical approximate solutions that provide decent accuracy. In this paper, we devise another way to obtain arbitrarily accurate numerical solutions to the general equation \eqref{eqn:MainEq0} using the collocation method. This numerical scheme is widely used to obtain numerical solutions to various differential and integral equations \cite{brunner2004collocation}. It has been successfully applied to the pantograph equation in \cite{liu1995linear} and later to its version with a general nonvanishing delay in \cite{brunner2011analysis}. In \cite{tian2020analysis}, authors considered a corresponding nonlinear problem and proved its solvability and convergence of the collocation scheme. 

The main result of this paper is the construction and proof of the convergence of a collocation numerical method to solve \eqref{eqn:MainEq0} (Section 3). Thanks to the auxiliary results on existence, uniqueness, and regularity (Section 2), we can obtain them with assumptions \textit{only} on the coefficients. The unique solution lies in the Lipschitz space, and in order to find the error of the numerical scheme, it is necessary to obtain some new results concerning polynomial interpolation for functions of that regularity. Numerical calculations (Section 4) verify that the method is robust, fast and able to also treat equations of weaker regularity such as H\"older continuous. In any case, it outperforms the classical Picard iteration with a computational cost of $O(n^2)$ with $n\rightarrow\infty$ being the number of degrees of freedom. There are still several open and interesting problems regarding this method, and we outline them in Section 5. 

\section{Existence, uniqueness, and regularity}
The space of Lipschitz functions is defined as follows
\begin{equation}
H^1[0,1] = \left\{v\in C[0,1]: \; \sup_{x,y \in[0,1]} \frac{|v(x)-v(y)|}{|x-y|}<\infty\right\}.
\end{equation}
Next, we introduce a closed subspace that becomes the Banach space in which we will look for solutions to the considered functional equation
\begin{equation}
H_0^1[0,1] = \left\{v\in C[0,1]: \; \sup_{x,y \in[0,1]} \frac{|v(x)-v(y)|}{|x-y|}<\infty \quad{\text{and}}\quad v(0) = v(1) = 0\right\},
\end{equation}
endowed with the natural norm
\begin{equation}
\|v\|:= \sup_{x,y \in[0,1]} \frac{|v(x)-v(y)|}{|x-y|}.
\end{equation}
We will frequently abbreviate the space symbols to $H^1$ or $H_0^1$ since the interval is understood. The space $H^1$ can also be endowed with a classical Lipschitz norm in the form $\|v\|_1 := |v(\xi)| + \|v\|$, where $\xi\in [0,1]$ is fixed. Note that with vanishing boundary conditions, that is in the space $H_0^1$ we have $\|v\|_1 = \|v\|$ because then we can take $\xi = 0$ or $\xi = 1$. We also denote the standard supremum norm by $\|\cdot\|_\infty$. 

Let us consider a general functional equation with boundary conditions
\begin{equation}\label{eqn:MainEqGeneral}
u(x) = Tu(x) + f(x), \quad x\in [0,1], \quad u(0) = 0, \quad u(1) = 0,
\end{equation}
where $T$ is a linear operator defined by
\begin{equation}
\label{eqn:TOperator}
Tu(x) = \varphi(x) u(\varphi_1(x)) + (1-\varphi(x)) u(\varphi_2(x)).
\end{equation}
For the coefficients we have to assume the following natural conditions 
\begin{equation}\tag{A}\label{eqn:Assumptions}
\begin{cases}
	\varphi \in H^1: & \varphi(0) = 0, \quad \varphi(1) = 1, \quad 0\leq \varphi(x)\leq 1, \\
	\varphi_{1},\varphi_{2} \in H^1: & \varphi_1(1) = 1, \quad \varphi_2(0) = 0, \quad 0\leq\varphi_{1}(x),\varphi_{2}(x)\leq 1,\\
	f \in H_0^1.
\end{cases}
\end{equation}
It is now a matter of simple calculation to show that $T: H_0^1 \mapsto H_0^1$. Furthermore, the action of the operator $T$ can be understood as taking a generalized convex combination of the function with mixed arguments. Also note that the assumption of vanishing boundary conditions is not limiting and can be done without any loss of generality. To wit, suppose that $u(0) = u_0$ and $u(1) = u_1$. Then, by introducing a new function
\begin{equation}
v(x) := u(x) - (1-x)u_0 - x u_1 = u(x) - h(x),
\end{equation}
it is straightforward to show that $v$ satisfies \eqref{eqn:MainEqGeneral} but with a different source function $f\mapsto f + Th-h \in H_0^1$ thanks to \eqref{eqn:Assumptions}. Therefore, an equation with general boundary conditions can be transformed into \eqref{eqn:MainEqGeneral}. 

The existence and uniqueness of the solution to \eqref{eqn:MainEqGeneral} can be established by imposing some growth and smoothness conditions on the continuity of the coefficients that guarantee the contractivity of $T$. 

\begin{thm}[Existence and uniqueness]\label{thm:Existence}
Assume \eqref{eqn:Assumptions} and suppose that $(1+\|\varphi\|)(\|\varphi_1\|+\|\varphi_2\|)<1$. Then, there exists an exactly one solution of \eqref{eqn:MainEqGeneral} in $H_0^1$ satisfying
\begin{equation}
	\|u\| \leq  \frac{\|f\|}{1-(1+\|\varphi\|)(\|\varphi_1\| + \|\varphi_2\|)}.
\end{equation}
\end{thm} 
\begin{proof}
Since $H_0^1$ is a complete Banach space we just have to show that $T$ is a contraction, because then the operator $(I-T)^{-1}$ is well-defined by the geometric series theorem (see \cite{zeidler2012applied}, Chapter 1.23) and the solution is given by 
\begin{equation}\label{eqn:SolutionExplicit}
	u(x) = (I-T)^{-1} f(x).
\end{equation}
The norm of the operator $T$ can be computed by noting the identity
\begin{equation}
	\begin{split}
		Tu(x)-Tu(y) 
		&= (\varphi(x) - \varphi(y)) u(\varphi_1(x)) + \varphi(y) (u(\varphi_1(x) - u(\varphi_1(y)))) \\
		&+ (\varphi(y) - \varphi(x)) u(\varphi_2(x))+(1-\varphi(y))(u(\varphi_2(x))-u(\varphi_2(y))),
	\end{split}
\end{equation}
where $x$, $y\in[0,1]$. Note that since $\varphi_1(1) = 1$ we can write
\begin{equation}
	|u(\varphi_1(x))| = |u(\varphi_1(x)) - u(\varphi_1(1))| = \frac{|u(\varphi_1(x)) - u(\varphi_1(1))|}{|\varphi_1(x)-\varphi_1(1)|} \frac{|\varphi_1(x)-\varphi_1(1)|}{1-x} (1-x) \leq \|u\| \|\varphi_1\|,
\end{equation}
and similarly for $\varphi_2$ with $x=0$. Therefore, taking into account that $0\leq \varphi(x) \leq 1$ we have
\begin{equation}
	\begin{split}
		\frac{|Tu(x)-Tu(y)|}{|x-y|} 
		&\leq (\|\varphi_1\| + \|\varphi_2\|) \|\varphi\| \|u\| + \frac{|u(\varphi_1(x) - u(\varphi_1(y))|}{|\varphi_1(x)-\varphi_1(y)|} \frac{|\varphi_1(x)-\varphi_1(y)|}{|x-y|} \\
		&+\frac{|u(\varphi_2(x) - u(\varphi_2(y))|}{|\varphi_2(x)-\varphi_2(y)|} \frac{|\varphi_2(x)-\varphi_2(y)|}{|x-y|} \leq (\|\varphi_1\| + \|\varphi_2\|) (1+\|\varphi\|) \|u\|,
	\end{split}
\end{equation}
hence, by taking the supremum, we obtain $\|T\| \leq (\|\varphi_1\| + \|\varphi_2\|) (1+\|\varphi\|)$ which is smaller than $1$ by the assumption. Therefore, $T$ is a contraction and the existence of a unique solution follows from the Banach fixed-point theorem (\cite{zeidler2012applied}, Theorem 1.A). Its norm can be estimated from \eqref{eqn:SolutionExplicit}
\begin{equation}
	\|u\| \leq \|(I-T)^{-1}\| \|f\| \leq \frac{\|f\|}{1-\|T\|} \leq \frac{\|f\|}{1-(1+\|\varphi\|)(\|\varphi_1\| + \|\varphi_2\|)},
\end{equation}
which ends the proof. 
\end{proof}

Provided that the coefficients are sufficiently regular, we can further prove that the solution retains the same smoothness. Note that for smooth functions with vanishing of \textit{either} boundary conditions (hence, in particular, for functions from $H_0^1$), the Lipschitz norm is equal to the supremum norm of the derivative. To wit, by the mean-value theorem (\cite{rudin1964principles}, Theorem 5.9), we have
\begin{equation}\label{eqn:LipschitzSupremum}
\|v\| = \sup_{x,y \in[0,1]} \frac{|v(x)-v(y)|}{|x-y|} = \sup_{\xi=\xi(x,y) \in [0,1]} |v'(\xi)| = \|v'\|_\infty.  
\end{equation}
The proof of the regularity theorem proceeds inductively and it is based on the following idea. First, we find the functional equation that has to be satisfied by the derivative of the solution (if there exists). Second, we form a sequence of difference quotients and use the Arzel\`a-Ascoli theorem (\cite{zeidler2012applied}, Chapter 1.11)  to show that it has a convergent subsequence. Third, we conclude that its limit must satisfy the initial functional equation with a unique solution. Therefore, the sequence of difference quotients actually converges to the derivative of the solution. Then the inductive step follows. 

\begin{prop}[Regularity]\label{prop:Regularity}
Assume \eqref{eqn:Assumptions} and let $\varphi_1,\varphi_2\in C^m[0,1]$ and $\varphi$, $f \in C^{m+1}[0,1]$ for some $m\geq 1$. Additionally, assume that $(1+\|\varphi\|)(\|\varphi_1\|+\|\varphi_2\|)<1$. Then, the solution $u$ of \eqref{eqn:MainEqGeneral} belongs to $C^m[0,1]$.
\end{prop}
\begin{proof}
Note that, by assumption, since $T$ is a contraction, according to Theorem \ref{thm:Existence} there exists a unique solution $u\in H_0^1$ to \eqref{eqn:MainEqGeneral}. We start by noticing that by differentiating \eqref{eqn:MainEqGeneral}, the prospective derivative $u'$ should be the solution of the following functional equation
\begin{equation}\label{eqn:DerivativeEquation}
	\begin{split}
		v &= \varphi(x) \varphi_1'(x) v(\varphi_1(x)) + (1-\varphi(x))\varphi_2'(x) v(\varphi_2(x)) + \varphi'(x) (u(\varphi_1(x))-u(\varphi_2(x)))) + f'(x) \\
		&=: T_1 v(x) + f_1(x),
	\end{split}
\end{equation}
where we defined a new operator $T_1: C[0,1] \mapsto C[0,1]$ and a source function $f_1$. Note that $T_1$ is a contraction since by the assumption we have
\begin{equation}
	\|T_1v\|_\infty \leq \varphi(x) \|\varphi_1'\|_\infty \|v\|_\infty + (1-\varphi(x)) \|\varphi_2'\|_\infty \|v\|_\infty,
\end{equation}
hence $\|T_1\|\leq \|\varphi_1'\|_\infty + \|\varphi_2'\|_\infty = \|\varphi_1\| + \|\varphi_2\|$ from \eqref{eqn:LipschitzSupremum}. Since, by the assumption we have $(1+\|\varphi\|)(\|\varphi_1\|+\|\varphi_2\|)<1$ and obviously $1+\|\varphi\|>1$, we deduce that  $\|T_1\|\leq \|\varphi_1\| + \|\varphi_2\|\leq 1$.

We have to show that the derivative $u'$ exists and is the solution of the above equation. To this end, we will show that the difference quotient $D_h u(x)$ defined by
\begin{equation}
	D_h u(x) := h^{-1}(u(x+h)-u(x)) \in H^1, 
\end{equation}
for sufficiently small $h$ such that $0\leq x+h \leq 1$, is convergent as $h\rightarrow 0$. Let us write \eqref{eqn:MainEqGeneral} for $x$ replaced by $x+h$, subtract, and use linearity to arrive at
\begin{equation}\label{eqn:DhEq}
	D_h u(x) = \varphi(x+h) D_h(u(\varphi_1(x))) + D_h\varphi(x) (u(\varphi_1(x))-u(\varphi_2(x))) + (1-\varphi(x+h)) D_h(u(\varphi_2(x))) + D_h f(x).
\end{equation}
Now, by applying the supremum norm, we can show that
\begin{equation}\label{eqn:DhBound}
	\|D_h u\|_\infty \leq (\|\varphi'_1\|_\infty + \|\varphi'_2\|_\infty)\|u\| + 2\|\varphi'\|_\infty \|u\|_\infty + \|f'\|_\infty =: M,
\end{equation}
which shows that the set $\{D_h u(x)\}_h$ is bounded in $C[0,1]$. Next, note that although $\|\cdot\|$ is not a norm on $H^1$ but a seminorm, the following are still true
\begin{equation}
	\|v w\| \leq \|v\|_\infty \|w\| + \|v\| \|w\|_\infty, \quad \|v\circ w\| \leq \|v\| \|w\|, \quad v, w\in H^1,
\end{equation}
along with the triangle inequality. Applying this (semi)norm onto \eqref{eqn:DhEq} and using the above relations we can write
\begin{equation}
	\begin{split}
		\|D_h u\| 
		&\leq \|\varphi\|_\infty \|D_h u\| \|\varphi_1\| + \|\varphi\| \|D_h u\|_\infty + \|1-\varphi\|_\infty \|D_h u\| \|\varphi_2\| + \|1-\varphi\| \|D_h u\|_\infty \\
		& + \|D_h \varphi\|_\infty \|u\circ \varphi_1 - u \circ \varphi_2\| + \|D_h \varphi\| \|u\circ \varphi_1 - u \circ \varphi_2\|_\infty + \|D_h f\|. 
	\end{split}
\end{equation}
Now, by the assumption \eqref{eqn:Assumptions} on the boundedness of $\varphi$\normalcolor, \eqref{eqn:LipschitzSupremum}, and \eqref{eqn:DhBound} we can further write
\begin{equation}
	\begin{split}
		\|D_h u\| 
		&\leq \|D_h u\| \|\varphi'_1\|_\infty + M\|\varphi'\|_\infty + \|D_h u\| \|\varphi'_2\|_\infty + M\|\varphi'\|_\infty \\
		& + M \|u\circ \varphi_1 - u \circ \varphi_2\| + \|D_h \varphi\| \|u\circ \varphi_1 - u \circ \varphi_2\|_\infty + \|D_h f\|. 
	\end{split}
\end{equation}
Furthermore, we have $\|u\circ \varphi_1\| \leq \|u\| \|\varphi_1\| = \|u\| \|\varphi'_1\|_\infty$ and similarly for $\varphi_2$. Additionally, $\|D_h f\| \leq \|f'\| = \|f''\|_\infty$ and the same for $\varphi$. Moreover, $\|u\circ \varphi_1\|_\infty \leq \|u\|_\infty$, whence
\begin{equation}
	\|D_h u\| \leq \left(\|\varphi_1'\|_\infty + \|\varphi_2'\|_\infty\right) \|D_h u\| + 2 M\|\varphi'\|_\infty + \left(\|\varphi_1'\|_\infty + \|\varphi_2'\|_\infty\right) M \|u\| + 2\|\varphi''\|_\infty \|u\|_\infty + \|f''\|_\infty,
\end{equation}
or by the assumption that $\|\varphi_1'\|_\infty + \|\varphi_2'\|_\infty < 1$ 
\begin{equation}
	\|D_h u\| \leq \frac{2 M\|\varphi'\|_\infty + \left(\|\varphi_1'\|_\infty + \|\varphi_2'\|_\infty\right) M \|u\| + 2\|\varphi''\|_\infty \|u\|_\infty + \|f''\|_\infty}{1-\|\varphi_1'\|_\infty - \|\varphi_2'\|_\infty}.
\end{equation}

Therefore, the sequence $\{D_h u(x)\}_h$ is equicontinuous. From the Arzel\`a-Ascoli theorem it has a subsequence that is uniformly convergent as $h\rightarrow 0$. Since the limit must satisfy \eqref{eqn:DerivativeEquation}, it must also be unique due to the contractivity of $T_1$. On the other hand, by the definition, the limit of the quotient is the derivative. Therefore, the sequence of difference quotients converges and its limit is actually the derivative $u'$ which means that $u\in C^1[0,1]$. 

The proof that $u\in C^m[0,1]$ for $m>1$ proceeds by induction and by exactly the same arguments. Therefore, we will omit the technical details. It is only important to note that in the inductive step the operator $T_{k+1}: C[0,1]\mapsto C[0,1]$ defined by differentiating the equation for $u^{(k)}$ is also a contraction of the form
\begin{equation}
	T_{k+1} v(x) = \varphi(x) (\varphi_1')^{k+1}(x) v(\varphi_1(x)) + (1-\varphi(x)) (\varphi_2')^{k+1}(x) v(\varphi_2(x)).
\end{equation}
Therefore, the norm is $\|T_{k+1}\|\leq \|\varphi_1'\|^{k+1}_\infty + \|\varphi_2'\|^{k+1}_\infty < 1$ which, by Banach fixed-point theorem, guarantees uniqueness of the solution to $v = T_{k+1} v + f_{k+1}$, where in $f_{k+1}$ we gather all terms resulting from differentiation that include all derivatives of order less than $k+1$ of $u$ and a combination of coefficients $\varphi$, $\varphi_{1},\varphi_{2}$, and $f$ (and all their derivatives). Then, we can form the respective difference quotient and use Arzel\`a-Ascoli theorem to prove the existence of its limit, which by uniqueness is equal to $u^{(k+1)}$. The process continues as long as $k\leq m$. 
\end{proof}

\section{Collocation method}
We are interested in finding a numerical approximation to the solution to  Problem \ref{eqn:MainEq0}. In \cite{turab2019analytic} it was suggested that Picard's iteration can be used for this task, provided the operator $T$ is a contraction. In \cite{paradiseFefiKishinLukas}, however, we show that although this might be a simple device to obtain some analytical approximations, its use in practice is very limited due to the very high computational cost. Thus, we must look for another numerical method that is robust, fast, and accurate. 

In a standard way, we introduce a uniform grid $x_i = i h$ with $h = 1/n$ and $0\leq i\leq n$. The choice of a uniform partition of the domain is chosen only for the simplicity of the presentation, and one can easily generalize to some other, maybe more efficient, nonuniform grids. The important point in constructing a numerical scheme is to note that even tough $x_i$ belongs to the grid, $\varphi(x_i)$ \textit{does not have to} be one of its points. This means that to discretize \eqref{eqn:MainEqGeneral} we have to construct an approximation $u_h$ that is defined throughout the domain and not only at the points of the grid. A natural candidate is to use the \textit{collocation method} (see \cite{brunner2004collocation}). Below, we construct a \textit{piecewise linear} collocation method and leave the analysis of higher-order approximations to future work. 

The numerical approximation to $u$ is constructed by approximating it by a continuous piecewise linear function $u_h$ that satisfies the equation \eqref{eqn:MainEqGeneral} at each point of the grid along with the boundary conditions. That is,
\begin{equation}\label{eqn:Collocation}
\begin{split}
	\begin{cases}
		u_h \text{ is continuous and is a linear polynomial on each }[x_{i-1}, x_{i}], \quad 1\leq i\leq n, \\
		u_h(0) = 0, \; u_h(1) = 0, & \text{(Boundary conditions)} \\
		u_h(x_i) = Tu_h(x_i) + f(x_i), \quad 1\leq i\leq n-1, &  \text{(Collocation)}\\ 
	\end{cases}
\end{split}
\end{equation}
The above definition is well-posed since, in total, we have $n$ linear polynomials to find, that is, $2n$ coefficients to determine. The continuity of each interface $x_i$ gives $n-1$ equations, which together with $n-1$ collocation points and two boundary conditions yield $n-1 + n-1 + 2 = 2n$ equations in total. The above numerical method can be written explicitly as a system of algebraic equations for on each subinterval we have
\begin{equation}
u_h(x) = a_i x + b_i, \quad x \in [x_{i-1},x_i], \quad 1\leq i\leq n.
\end{equation}
Therefore, taking into account \eqref{eqn:Collocation} for $1\leq i\leq n-1$ we have
\begin{equation}\label{eqn:AlgebraicSystem}
\begin{cases}
	a_{i-1} x_i + b_{i-1} = a_i x_i + b_i, & \text{(Continuity)}, \\
	a_i x_i + b_i = \varphi(x_i) u(\varphi_1(x_i)) + (1-\varphi(x_i)) u(\varphi_2(x_i)) + f(x_i), & \text{(Collocation)}, \\
	b_0 = 0, \quad a_n + b_n = 0, & \text{(Boundary conditions)},
\end{cases}
\end{equation}
which gives us the correct number of equations to solve. Even in the most general case this can be computed much cheaper than iterating the main functional equation.

\begin{rem}
Note that since, in general, $\varphi_{1}(x_i),\varphi_{2}(x_i)$ may not lie in the subinterval $[x_{i-1}, x_i]$ the resulting system matrix \eqref{eqn:AlgebraicSystem} may not be symmetric. It seems that showing the invertibility of such a matrix can be very difficult. In \cite{brunner2011analysis} authors proved this claim for a simpler equation $u(x) = a u(b x) + f(x)$ with a sophisticated and involved analysis. It remains an open problem to rigorously show the solvability of \eqref{eqn:AlgebraicSystem} in the general case. 
\end{rem}

We can proceed to prove the convergence of the collocation approximation. It is very useful to understand it as a certain projection. In the result below, we prove its properties in the Lipschitz space. Note that quantitative information about the convergence rate strongly depends on the smoothness of the projected function.  

\begin{lem}\label{lem:Projection}
Let $P_h$ be the interpolation projection operator defined on each subinterval $[x_{i-1}, x_i]$ with $1\leq i\leq n-1$ by
\begin{equation}\label{eqn:Projection}
	P_h v(x) = \frac{1}{h} \left((x-x_{i-1})v(x_i) + (x_i-x) v(x_{i-1})\right), \quad x \in [x_{i-1}, x_i], \quad v \in H_0^1[0,1].
\end{equation}
Then, $P_h: H_0^1[0,1] \mapsto H_0^1[0,1]$ with $\|P_h\| = 1$. Furthermore, the interpolation error vanishes in the limit, that is,
\begin{equation}\label{eqn:InterpolationError}
	\lim\limits_{h\rightarrow 0}\|v - P_h v\| = 0.
\end{equation}
If, additionally, $v\in C^2[0,1]$, then we have the stronger assertion that
\begin{equation}\label{eqn:ErrorLipschitz}
	\|v - P_h v\| \leq Ch,
\end{equation}
for some $C=C(v) > 0$. 
\end{lem}
\begin{proof}
We have to prove that $P_h v \in H_0^1[0,1]$ provided that $v \in H_0^1[0,1]$. Note that by definition $P_h v$ is piecewise continuous. Since in each subinterval $[x_{i-1},x_i]$ with $1\leq i\leq n-1$ the function $P_h v$ is linear, it is also piecewise Lipschitz and even piecewise differentiable. The Lipschitz constant on each subinterval is then equal to $|v(x_i)-v(x_{i-1})|/h$. To show that $P_h v$ is Lipschitz globally on $[0,1]$ we have to deal with the grid points, at which the function is not differentiable. To overcome this difficulty, we will mollify $P_h v$ with a positive function $\psi\in C_c^\infty[0,1]$ satisfying $\int_0^1 \psi(x) dx = 1$ (for details see \cite{evans2022partial}, Appendix C.5.). More specifically, for $\epsilon>0$ we define the mollification operator $J_\epsilon$ with the formula
\begin{equation}
	J_\epsilon w(x) := \frac{1}{\epsilon} \int_0^1 \psi\left(\frac{x-y}{\epsilon}\right) w(y) dy, \quad w \in L^1(0,1).
\end{equation}
By a standard theory, we know that $J_\epsilon w$ converges uniformly to $w$ on $[0,1]$ as $\epsilon\rightarrow 0^+$. Now, since $J_\epsilon P_h v$ is differentiable, we can use the mean value theorem to obtain
\begin{equation}
	|J_\epsilon P_h v(x) - J_\epsilon P_h v(y)| = |(J_\epsilon P_h v)'(\xi)| |x-y|,
\end{equation}
for some $\xi \in (0,1)$. From this we can provide a bound for the derivative. It is natural to split the integral and integrate by parts
\begin{equation}
	\begin{split}
		(J_\epsilon P_h v)'(\xi) 
		&= \sum_{i=1}^{n} \frac{1}{\epsilon^2} \int_{x_{i-1}}^{x_i} \psi'\left(\frac{\xi-y}{\epsilon}\right) (P_h v)(y) dy = \sum_{i=1}^{n} \left[-\frac{1}{\epsilon}\psi\left(\frac{\xi-y}{\epsilon}\right) (P_h v)(y)\right]_{y = x_{i-1}}^{y=x_i} \\
		&+ \frac{1}{\epsilon}\sum_{i=1}^{n}\int_{x_{i-1}}^{x_i}\psi\left(\frac{\xi-y}{\epsilon}\right) (P_h v)'(y) dy.
	\end{split}
\end{equation}
The first sum is telescoping since $P_h v$, and thus its limits on each side of $x_i$ coincide. Moreover, the derivative of $P_h v$ is equal to  $(v(x_i)-v(x_{i-1}))/h$, hence
\begin{equation}
	\begin{split}
		(J_\epsilon P_h v)'(\xi) &= \frac{1}{\epsilon}\left(\psi\left(\frac{\xi}{\epsilon}\right) (P_h v)(0)-\psi\left(\frac{\xi-1}{\epsilon}\right) (P_h v)(1)\right) \\
		& + \sum_{i=1}^{n}\frac{v(x_i)-v(x_{i-1})}{h}\frac{1}{\epsilon} \int_{x_{i-1}}^{x_i}\psi\left(\frac{\xi-y}{\epsilon}\right) dy.
	\end{split}
\end{equation}
By noticing that the $P_h v(0) = v(0) = 0$ and $P_h v(1) = v(1) = 0$ we arrive at the bound
\begin{equation}
	|(J_\epsilon P_h v)'(\xi)| \leq \max_{1\leq i\leq n} \left|\frac{v(x_i)-v(x_{i-1})}{h}\right| \frac{1}{\epsilon}\int_0^1 \psi\left(\frac{\xi-y}{\epsilon}\right)dy \leq \|v\|,
\end{equation}
and, therefore, we obtain the fundamental bound independent on $\epsilon$
\begin{equation}
	|J_\epsilon P_h v(x) - J_\epsilon P_h v(y)| \leq \|v\| |x-y|,
\end{equation}
which implies
\begin{equation}
	|P_h v(x) - P_h v(y)| \leq \|v\| |x-y|.
\end{equation}
Hence, by taking the supremum over $x,y\in [0,1]$, we see that $P_h v$ is Lipschitz with the norm $\|v\|$ and $\|P_h\| \leq 1$. This, along with the fact that $P_h$ is a projection, yields $\|P_h\|=1$. 

We can use the same mollifier method as above to establish that the interpolation error in the Lipschitz norm vanishes as $h\rightarrow 0$. Let $e_h = v - P_h v$ be the interpolation error. We then have
\begin{equation}
	\|e_h\| \leq \|e_h - J_\epsilon e_h\| + \|J_\epsilon e_h\|.
\end{equation}
Note that since $e_h$ is a Lipschitz function it belongs to the Sobolev space $W_{1,\infty}$ and $\|e_h-J_\epsilon e_h\| = \|e_h-J_\epsilon e_h\|_{W^{1,\infty}}$ (see \cite{evans2022partial}, Chapter 5.8.). Also, from a standard theory of mollifiers we know that the convergence is also in that norm, hence the first term above vanishes as $\epsilon\rightarrow 0$. For the second term, we first estimate the error in the supremum. To wit, from \eqref{eqn:Projection} and for $x\in[x_{i-1}, x_i]$ we have
\begin{equation}
	\begin{split}
		|e_h(x)| = |v(x) - P_h v(x)| 
		&= \frac{1}{h} |(x-x_{i-1})(v(x_i)-v(x)) + (x_{i} - x)(v(x_{i-1})-v(x))| \\
		&\leq \frac{(x-x_{i-1})(x_{i} - x)}{h} \left(\left|\frac{v(x_i)-v(x)}{x_i-x}\right| + \left|\frac{v(x)-v(x_{i-1})}{x-x_{i-1}}\right|\right) \\
		&\leq 2 \|v\| h,
	\end{split}
\end{equation}
which follows from the fact that $v\in H_0^1$. Hence, $\|e_h\|_\infty \leq 2 \|v\| h$. Now, the Lipschitz norm of the mollified error can be estimated by the use of the mean-value theorem
\begin{equation}
	|J_\epsilon e_h (x) - J_\epsilon e_h(y)| = |(J_\epsilon e_h)'(\xi)| |x-y|,  
\end{equation}
for some $\xi$ between $x$ and $y$ belonging to $[0,1]$. Therefore, $\|J_\epsilon e_h\| \leq \|(J_\epsilon e_h)'\|_\infty$, and 
\begin{equation}
	|(J_\epsilon e_h)'(\xi)| \leq \frac{1}{\epsilon^2} \int_0^1 \left|\psi'\left(\frac{\xi-z}{\epsilon}\right)\right| |e_h(z)| dz \leq \frac{\|e_h\|_\infty}{\epsilon} \int_\frac{\xi-1}{\epsilon}^{\frac{\xi}{\epsilon}} |\psi'(s)| ds \leq \frac{\|e_h\|_\infty}{\epsilon} \int_{0}^{1} |\psi'(s)| ds \leq \frac{C h}{\epsilon},
\end{equation}
by the change of the variable $s = (\xi-z)/\epsilon$ with a constant $C$ dependent only on the mollifier $G$. Therefore, we arrive at the following
\begin{equation}
	\|e_h\| \leq \|e_h - J_\epsilon e_h\| + \frac{C h}{\epsilon},
\end{equation}
which vanishes as $h\rightarrow 0$ provided we take $\epsilon = h^{\delta}$ for some $0<\delta<1$. Therefore, the projection error vanishes as we refine the grid. 

Finally, we obtain the error bounds for the case where $v \in C^2[0,1]$. Since it is very convenient to use the mollification, we once again have the following $|J_\epsilon e_h(x) - J_\epsilon e_h(y)| = |(J_\epsilon e_h)'(\xi)||x-y|$, and
\begin{equation}
	|(J_\epsilon e_h)'(\xi)| \leq \frac{1}{\epsilon} \sum_{i=1}^n \int_{x_{i-1}}^{x_i} \psi\left(\frac{\xi-z}{\epsilon}\right) |(e_h)'(z)| dz \leq \|(e_h)'\|_\infty,
\end{equation}
where, as above, have integrated by parts and noticed that the limit terms vanish due to boundary conditions. Due to standard estimates using Taylor series, we can show that $\|(e_h)'\|_\infty = \|(v - P_h v)'\|_\infty \leq C h$, which concludes the proof. 
\end{proof}

With the use of the above lemma, we can obtain the convergence of the collocation method in a simple way. In the following, we use the typical method of projecting the equation onto a finite-dimensional space and estimating the error therein.

\begin{thm}[Convergence]\label{thm:Convergence}
Assume \eqref{eqn:Assumptions} and $(1+\|\varphi\|)(\|\varphi_1\|+\|\varphi_2\|)<1$. Let $u$ be the solution of \eqref{eqn:MainEqGeneral} and suppose that \eqref{eqn:AlgebraicSystem} has a unique solution. Then, for $u_h$ being the collocation approximation to $u$, we have
\begin{equation}
	\|u-u_h\| \leq \frac{1}{1-\|T\|}\|u-P_h u\|.
\end{equation}
If, additionally, $u\in C^2[0,1]$, then
\begin{equation}
	\|u-u_h\|_\infty \leq C h^2,
\end{equation}
for some $C=C(u)>0$. 
\end{thm}
\begin{proof}
We can project equation \eqref{eqn:Collocation} onto the linear interpolation space to obtain
\begin{equation}
	u_h = P_h T u_h + P_h f,
\end{equation}
since $P_h u_h = u_h$. Then, applying the same projection on $u = Tu + f$ and subtracting, we obtain
\begin{equation}
	P_h T (u - u_h) = P_h u - u_h = P_h u - u + u - u_h, 
\end{equation}
or
\begin{equation}
	(I-P_h T) (u-u_h) = P_h u - u.
\end{equation}
Now, since by our assumptions, Theorem \ref{thm:Existence} yields $\|T\| < 1$. Furthermore, by Lemma \ref{lem:Projection} the operator $(I-P_h T)^{-1}$ exists due to the the geometric series theorem (see \cite{zeidler2012applied}, Chapter 1.23) and we can write
\begin{equation}
	\|u-u_h\| \leq \|(I-P_h T)^{-1}\| \|P_h u - u\| \leq \frac{1}{1-\|T\|} \|P_h u - u\|. 
\end{equation}
Finally, in the case of $u\in C^2[0,1]$ we can see that the error of the interpolation $P_h u - u$ in the Lipschitz norm is given by \eqref{eqn:ErrorLipschitz}. Therefore, for $x\in[x_{i-1},x_i]$ we have
\begin{equation}
	|P_h u(x) - u(x)| = \left| \int_{x_{i-1}}^{x} \left(\left(P_h u\right)'(z) - u'(z)\right) dz\right| \leq C h \int_{x_{i-1}}^{x_i} dz = C h^2,
\end{equation}
what concludes the proof after taking the supremum over $x\in[0,1]$. 
\end{proof}

\begin{rem}
There is an open problem to prove that the general $m$-th order collocation method is convergent for the problem \eqref{eqn:MainEqGeneral}. More specifically, we subdivide each interval by a fine mesh of collocation points, that is, set $x_{i,j} := x_i + c_j h$, where $0\leq c_0 < c_1 < c_2 < ... < c_m \leq 1$. Since we now have $m+1$ points in each interval $[x_{i-1}, x_{i}]$ we can look for a polynomial of degree $m$ that collocates our equation \eqref{eqn:MainEqGeneral} and is continuous. It remains an open problem to show results analogous to Lemma \ref{lem:Projection} or to verify whether $P_h T$ converges in the norm to $T$. 
\end{rem}

\section{Numerical examples}
We have conducted several different numerical experiments with the above collocation method. In the following, we illustrate the scheme with several examples. All computations have been conducted in \texttt{Mathematica} symbolic manipulation environment. 

\subsection{Paradise fish equation}
We consider a paradise fish equation that was developed in \cite{turab2019analytic} and has the following form
\begin{equation}
v(x) = x v(1-\alpha + \alpha x) + (1-x) v(\beta x), \quad 0<\alpha\leq \beta < 1, \quad v(0) = 0, \quad v(1) = 1.
\end{equation}
In the experiment conducted in \cite{bush1956two} and mathematically modeled in \cite{turab2019analytic}, the fish were given two gates through which to choose to swim. One of them resulted in a higher probability of obtaining a reward than the other. Therefore, according to the reinforced-extinction model the probability of choosing the rewarded gate should increase in the subsequent trial, and the probability of swimming through the unrewarded gate decreases. If $x$ is the initial probability of choosing the rewarded gate, in the subsequent trial it will increase to $1-(1-\alpha)x$, while the other decreases to $\beta x$. Therefore, $\alpha$ and $\beta$ can be considered as learning rates.

Note that the above is not of the form \eqref{eqn:MainEqGeneral} due to nonzero boundary conditions. However, as mentioned above, we can put $u(x) = v(x) - x$. Then, the new function satisfies the following equivalent functional equation with vanishing boundary conditions
\begin{equation}\label{eqn:ParadiseFish}
u(x) = x u(1-\alpha + \alpha x) + (1-x) u(\beta x) + (\beta-\alpha)(1-x)x, \quad 0<\alpha\leq \beta < 1, \quad u(0) = 0, \quad u(1) = 0.
\end{equation}
Note that from the above we immediately see that if $\alpha=\beta$ then there exists only a trivial solution (that is, $v(x) = x$). An exemplary solution plot is presented in Fig. \ref{fig:FishExemplary}. We can see that the closer $\alpha$ is to $\beta$ the smaller the solution. On the other hand, for a large difference between parameters, the function $u$ develops a pronounced maximum near $x=0$. 

\begin{figure}
\centering
\includegraphics[scale = 0.9]{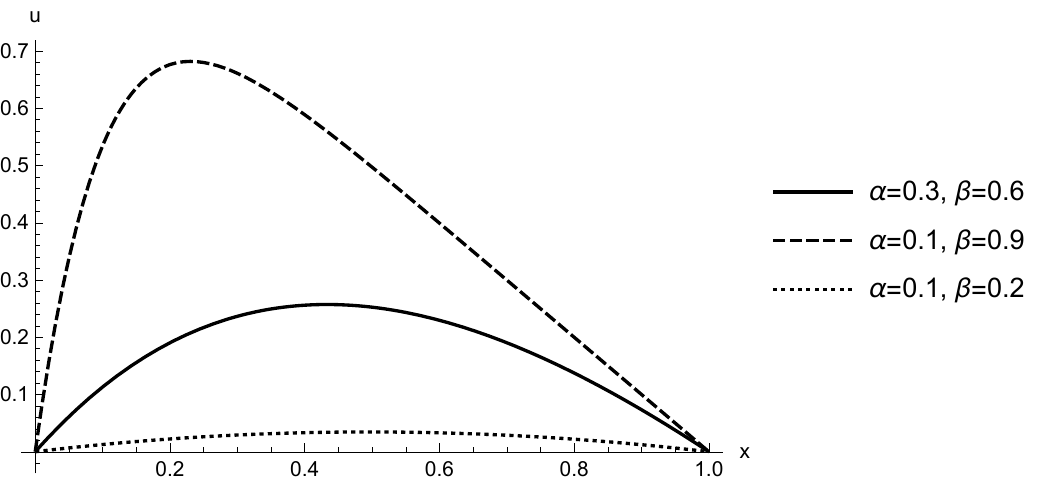}
\caption{An exemplary of of the solution to \eqref{eqn:ParadiseFish} obtained collocation method with the step $h=10^{-2}$. }
\label{fig:FishExemplary}
\end{figure}

Since the analytical form of the exact solution is not available, we cannot directly compute the error. However, to estimate the order of convergence, we can use extrapolation. That is, we can estimate the order of convergence by comparing the numerical solution with those computed using a finer mesh. The formula is the following
\begin{equation}
\text{order of convergence }\approx \log_2 \frac{\|u_h-u_{\frac{h}{2}}\|_\infty}{\|u_\frac{h}{2}-u_{\frac{h}{4}}\|_\infty}.
\end{equation}
The results of our calculations are gathered in the Tab. \ref{tab:fishOrders}. As we can see, the order computer in the supremum norm is consistent with the estimate given in Theorem \ref{thm:Convergence}. The convergence is faster for smaller values of $\alpha$ and $\beta$ but, overall, we can conclude that the method is second order accurate. This observation is consistent with Proposition \ref{prop:Regularity} which implies sufficient regularity for $(1+\|\varphi\|)(\|\varphi_1\|+\|\varphi_2\|) = 2 (\alpha + \beta) < 1$. However, as numerical computations indicate, the convergence is second-order for all admissible values of parameters. 

\begin{table}
\centering
\begin{tabular}{ccccccccc}
	\toprule
	$\alpha/\beta$ & 0.2 & 0.3 & 0.4 & 0.5 & 0.6 & 0.7 & 0.8 & 0.9 \\
	\midrule
	$0.1$ & 2.80 & 2.98 & 1.80 & 2.20 & 1.92 & 2.09 & 1.77 & 1.86  \\
	$0.2$ & -- & 2.45 & 1.67 & 2.08 & 1.90 & 2.05 & 1.77 & 1.87  \\
	$0.3$ & -- & -- & 1.91 & 2.15 & 1.97 & 2.06 & 1.78 & 1.87  \\
	$0.4$ & -- & -- & -- & 1.86 & 1.76 & 1.95 & 1.76 & 1.86  \\
	$0.5$ & -- & -- & -- & -- & 1.95 & 2.00 & 1.80 & 1.86  \\
	$0.6$ & -- & -- & -- & -- & -- & 1.93 & 1.80 & 1.88  \\
	$0.7$ & -- & -- & -- & -- & -- & -- & 1.89 & 1.88  \\
	$0.8$ & -- & -- & -- & -- & -- & -- & -- & 1.87  \\
	\bottomrule
\end{tabular}
\caption{Estimated orders of convergence of the collocation method applied to \eqref{eqn:ParadiseFish} for different values of $\alpha$ and $\beta$. The base step for the calculation is $h=2^{-8}$. }
\label{tab:fishOrders}
\end{table}

\subsection{Exact solution (smooth)}
To investigate further the performance of the collocation method, we tested it on an engineered equation with a known exact solution. To this end, we investigate the nonlocal and nonlinear functional equation \eqref{eqn:MainEqGeneral} for which we choose

\begin{equation}\label{eqn:ExactProblem}
\varphi(x) = x^2, \quad \varphi_1(x) = 1-\frac{\alpha}{2}(1-x), \quad \varphi_2(x) = 1-e^{-\frac{\alpha}{2}x}, \quad \alpha \in \left(0, \frac{1}{3}\right).
\end{equation}
In order for the exact solution to be equal to
\begin{equation}
u(x) = \sin\left(\pi x\right) \in C^2[0,1],
\end{equation}
we obviously have to put $f(x) := u(x) - Tu(x)$. The assumptions of Theorem \ref{thm:Convergence} are satisfied since the coefficients have appropriate boundary conditions and
\begin{equation}
(1+\|\varphi\|)(\|\varphi_1\| + \|\varphi_2\|) = (1 + 2)\left(\frac{\alpha}{2}+\frac{\alpha}{2}\right) = 3 \alpha < 1.
\end{equation}
Now, we can compute the error directly and the results are presented in Fig. \ref{fig:ExactError}. As can be seen, the error decreases as $h^2 = n^{-2}$, yielding the second order of convergence. We have also tested the collocation scheme with many different equations and obtained essentially the same results. We can claim that the scheme is second-order convergent. 

\begin{figure}
\centering
\includegraphics[scale = 1]{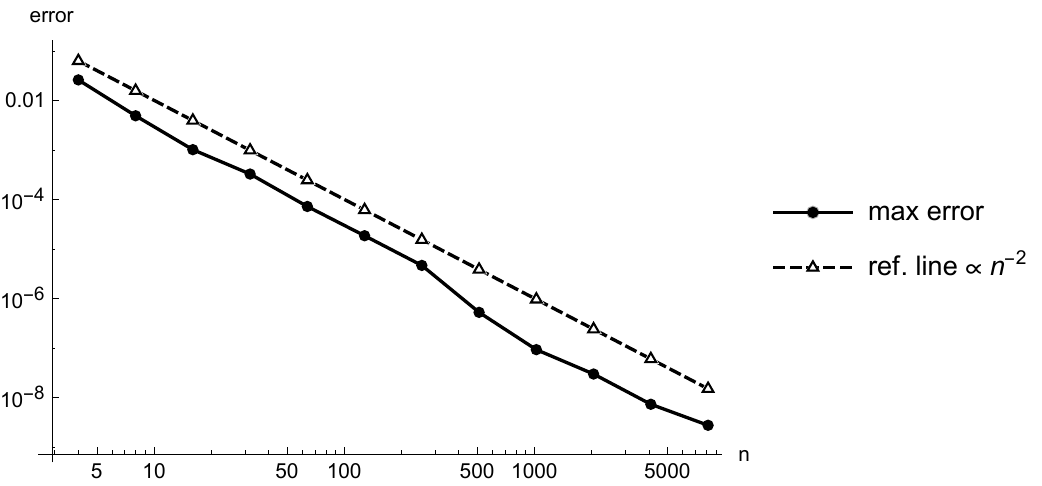}
\caption{Error (solid line) of the collocation method applied to the problem \eqref{eqn:ExactProblem} with respect to the number of subdivisions of the interval $n$. The reference line $n^{-2}$ is depicted with a dashed line and $\alpha = 0.3$. }
\label{fig:ExactError}
\end{figure} 

We have performed several benchmarks of the collocation method and compared it to the standard iteration scheme for \eqref{eqn:MainEqGeneral}. An exemplary graph of the computation time versus error is depicted in Fig. \ref{fig:timesError} where we solved problem \eqref{eqn:ExactProblem} for different discretization parameters. In the iteration scheme, we have plotted the time needed to obtain the subsequent term via recursion. As can be seen, first iterates yield almost no increase in the computation time, however, later the complexity increases drastically. However, the collocation method yields a stable increase in the needed cost and stays in a fast and accurate regime.  

\begin{figure}
\centering
\includegraphics[scale = 1]{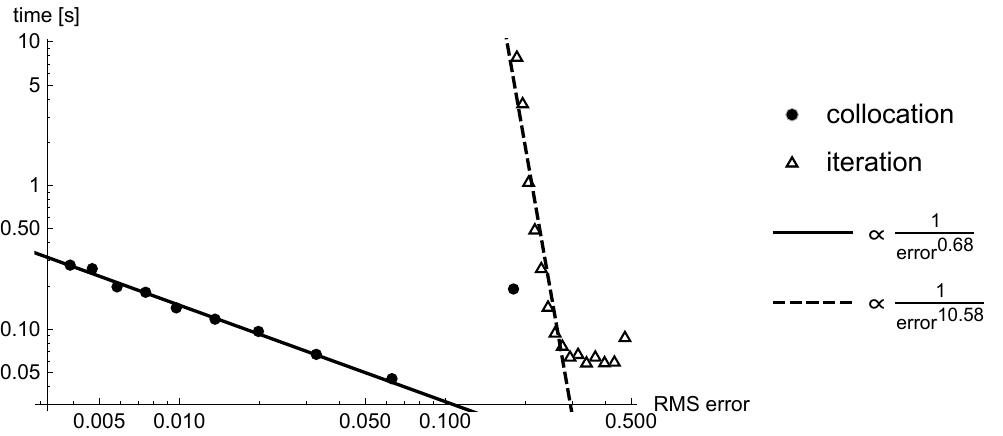}
\caption{Computation time as a function of the $L^2$ (RMS) error on the log-log scale. Lines represent the fitted trends.}
\label{fig:timesError}
\end{figure} 

We have also performed some absolute computation time benchmarks to determine the collocation solution to several problems with various $\varphi$, $\varphi_{1}$, $\varphi_{2}$, and $f$. An exemplary graph of these results is presented in Fig. \ref{fig:timesEvaluation}. Here, we present a log-log plot of the computation time needed to obtain a collocation approximation with $n$ degrees of freedom (subintervals). As can be seen, the numerical solution can be obtained in a polynomial time, specifically in $O(n^{2.07})$ as $n\rightarrow\infty$. Computations for different functional equations yielded similar results. We can conclude that in practice the algebraic system \eqref{eqn:AlgebraicSystem} can be solved cheaply and efficiently since its cost always stays below the usual Gaussian elimination complexity $O(n^3)$. 

\begin{figure}
\centering
\includegraphics[scale = 1]{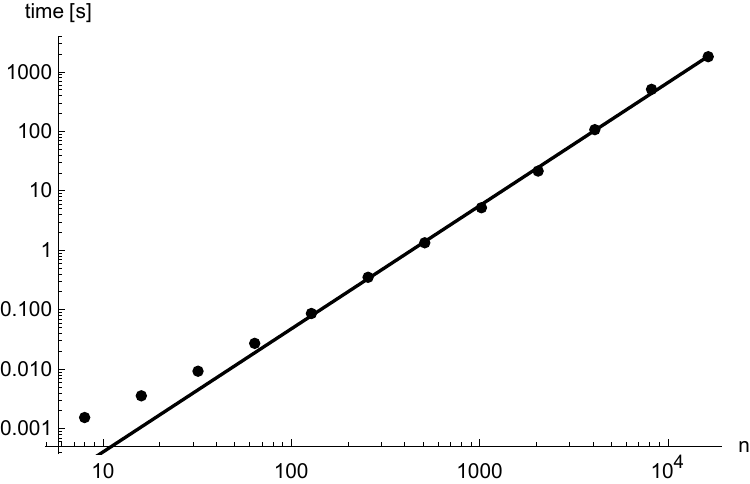}
\caption{Computation time as a function of the number of degrees of freedom $n$. The solid line represents the fitted trend $\propto n^{2.07}$.}
\label{fig:timesEvaluation}
\end{figure} 

\subsection{Exact solution (nonsmooth)}
Although our theory presented in the previous section requires at least Lipschitz regularity, we can explore the performance of the collocation method for functions of less smoothness. As an example, let us choose the following coefficients

\begin{equation}\label{eqn:ExactProblemNonsmooth}
\varphi(x) = x, \quad \varphi_1(x) = 1-\frac{\alpha}{2}(1-x), \quad \varphi_2(x) = \frac{\alpha}{2}x, \quad \alpha \in \left(0, \frac{1}{2}\right).
\end{equation}
And choose $f(x)$ such that the exact solution is equal to
\begin{equation}
u(x) = \sqrt{\frac{1}{2}-\left|x-\frac{1}{2}\right|}.
\end{equation}
By simple calculations, it can be shown that the above function is H\"older continuous with order $1/2$ and the assumption $(1+\|\varphi\|)(\|\varphi_1\|+\|\varphi_2\|)<1$ is satisfied. The graph of the function, with a characteristic cusp in the middle, is shown in Fig. \ref{fig:Nonsmooth}. Notice that even taking a relatively small step $h = 2^{-8}$ we obtain a low-accuracy approximation to the exact solution. This phenomenon is expected as the solution is not regular enough. As can be inferred from Fig. \ref{fig:ExactErrorNonsmooth} the order of the collocation method is equal to $1/2$ - the same as the H\"older regularity exponent (compare the values of the vertical axis with Fig. \ref{fig:ExactError}). We can conclude that the numerical scheme is applicable also in situations of low regularity and we leave the investigation concerning a rigorous proof of that fact for future work. 

\begin{figure}
\centering
\includegraphics[scale = 1]{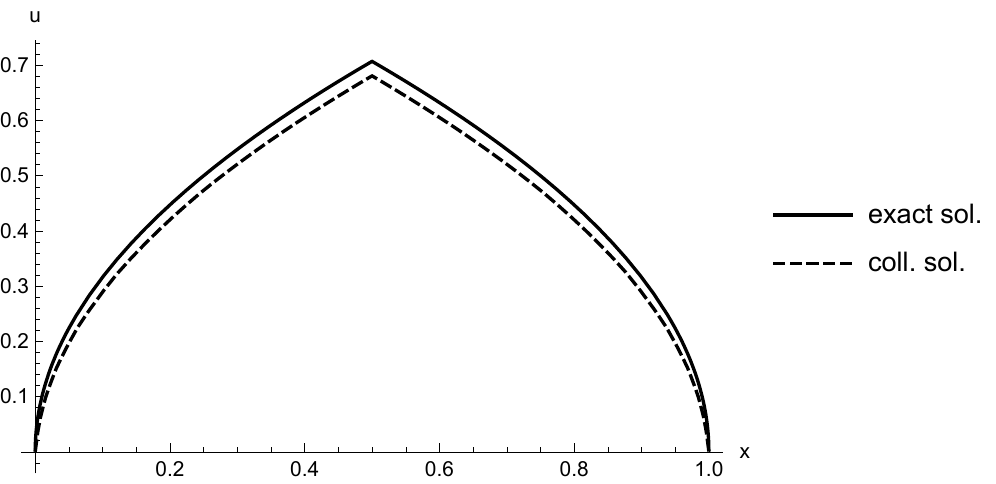}
\caption{Exact solution to \eqref{eqn:ExactProblemNonsmooth} (solid line) obtained collocation method with the step $h=2^{-8}$ (dashed line) and $\alpha = 0.45$.}
\label{fig:Nonsmooth}
\end{figure}

\begin{figure}
\centering
\includegraphics[scale = 1]{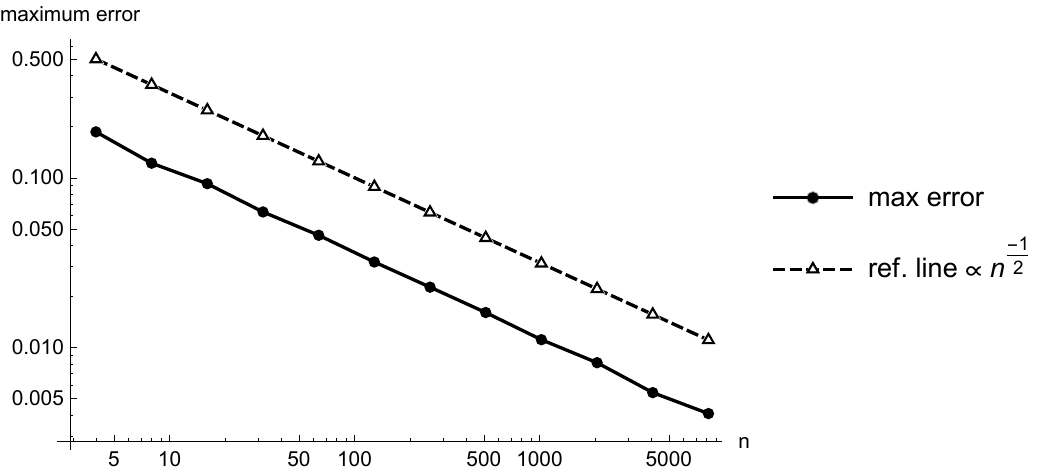}
\caption{Error (solid line) of the collocation method applied to the problem \eqref{eqn:ExactProblemNonsmooth} with respect to the number of subdivisions of the interval $n$. The reference line $n^{-1/2}$ is depicted with a dashed line. Here, $\alpha = 0.45$. }
\label{fig:ExactErrorNonsmooth}
\end{figure} 

\section{Conclusion and future work}
We have developed a second-order convergent numerical method to solve a general class of nonlocal functional equations of the form \eqref{eqn:MainEqGeneral}. Under some mild regularity assumptions, the solution exists and is unique. Moreover, if we allow for coefficients of higher smoothness, we can prove that the corresponding solution inherits the same regularity. 

The iteration method is seemingly natural way to obtain approximate solutions to \eqref{eqn:MainEqGeneral}, however, its lacks computational efficiency. The collocation method provides a fast and accurate alternative, even in the piecewise-linear second-order scheme. During our research, several open problems and questions have arisen that will be interesting to tackle in the future. In particular, it remains to show that the finite-dimensional system \eqref{eqn:AlgebraicSystem} has a unique solution, whether it is possible to prove that a general $m$-th order collocation method is convergent and to find error estimates in the H\"older continuous case. 

\section*{Acknowledgements}
\L.P. has been supported by the National Science Centre, Poland (NCN) under the grant Sonata Bis with a number NCN 2020/38/E/ST1/00153.


\end{document}